\documentclass[english,book]{amsart}
\usepackage{amssymb,amscd,amsmath,amsthm}

\usepackage[full]{textcomp} 
\usepackage[p,osf]{fbb} 
\usepackage[scaled=.95,type1]{cabin} 
\usepackage[varqu,varl]{zi4}

\usepackage[libertine]{newtxmath}

\usepackage[english]{babel}
\usepackage{csquotes}
\usepackage{todonotes}
\usepackage[colorlinks=true]{hyperref}
\usepackage{bookmark}
\usepackage{dutchcal}
\usepackage{eufrak}
\usepackage[shortlabels]{enumitem}
\usepackage{ dsfont }
\usepackage{mathtools}
\usepackage{mathbbol}
\usepackage{stmaryrd}
\usepackage{float}
\usepackage{tikz}
\usetikzlibrary{cd}
\usetikzlibrary{patterns}
\usepackage{moreenum}
\usepackage{animate}
\usepackage{array}
\usepackage{phoenician}
\usepackage{xstring}
\usepackage{xspace}
\usepackage{adjustbox}
\usepackage{wasysym}
\usepackage{marvosym}
\usepackage{manfnt}
\usepackage[noabbrev]{cleveref}
\usepackage{quiver}

\newcommand{\field}[1]{\mathbb{#1}}
\newcommand{\mc}[1]{\ensuremath{\mathcal{#1}}}
\newcommand{\ul}[1]{\underline{#1}}
\newcommand{\wt}[1]{\widetilde{#1}}

\theoremstyle{definition}
\newtheorem{definition}{\ul{Definition}}[subsection]

\newtheorem{theorem}[definition]{\ul{Theorem}}
\newtheorem*{theorem*}{\ul{Theorem}}
\newtheorem{proposition}[definition]{\ul{Proposition}}
\newtheorem{corollary}[definition]{\ul{Corollary}}
\newtheorem{lemma}[definition]{\ul{Lemma}}

\theoremstyle{definition}
\newtheorem{construction}[definition]{\ul{Construction}}
\newtheorem{notation}[definition]{\ul{Notation}}

\newtheorem*{problem*}{\ul{Problem}}
\newtheorem{example}[definition]{\ul{Example}}

\theoremstyle{remark}
\newtheorem{remark}[definition]{\ul{Remark}}

\def\DD{\mathbf{\Delta}}

\def\C{\field{C}}

\def\CP{\field{CP}}

\def\sse{\subseteq}

\def\cd{\cdot}

\def\R{\field{R}}

\def\O{\ensuremath{\mathcal{O}}}

\def\Lc{\mc{L}}

\def\vp{\varphi}

\def\SS{\mathbb{S}}
\def\DD{\mathbb{D}}

\newcolumntype{L}{>{$}l<{$}}
\def\temp{&}
\catcode`&=\active
\let&=\temp

\def\CT{\C^\times}

\def\UGQT{\texttt{UGQT}\xspace}

\Crefname{construction}{Construction}{Constructions}	

\def\doi#1{\href{https://doi.org/#1}{doi:#1}}
\def\arXiv#1{\href{https://arxiv.org/abs/#1}{arXiv:#1}}
\def\web#1{\href{#1}{web}}

\makeatletter
\newcommand\thankssymb[1]{\textsuperscript{\@fnsymbol{#1}}}

\begin{document}
\title[Chern classes]{Chern Classes of Toric Variety Bundles}
\author[G. Taroyan]{Gregory Taroyan\thankssymb{2}}
\thanks{\thankssymb{2}The author's work is supported by Vanier Canada Graduate Scholarship, funding number CGV --- 192668.}
\begin{abstract}
	In this paper, we resolve a conjecture of Khovanskii--Monin on the Chern classes of toric variety bundles. The main result is a formula for the total Chern class of the tangent bundle of a toric variety bundle in terms of the total Chern class of the base and the total Chern class of the toric fibre. The result serves as a simultaneous generalization of the description of the total Chern class of a projectivized vector bundle and of the formula for the total Chern class of a toric variety in terms of the Chern classes of the toric divisors. We also establish a topological version of this statement for stably complex quasitoric manifolds. As an immediate application, we obtain a formula for the total Chern class of a toroidal horospherical variety in terms of the Chern classes of the generalized flag variety and the total Chern class of the toric fibre, as well as a new proof of Masuda's formula for equivariant Chern classes. This paper is written with a view towards finding polytopal models for various numeric invariants of spherical varieties.
\end{abstract}
\maketitle
\setcounter{tocdepth}{1}
\tableofcontents
\setcounter{tocdepth}{3}
\section{Introduction}
In this text, we resolve in the affirmative a conjecture due to Khovanskii--Monin \cite[Conjecture 3.6]{Khovanskii_Monin} on the Chern classes of toric variety bundles. Toric variety bundles should be regarded as a relative version of toric varieties obtained by applying the associated bundle construction to a principal torus bundle over a given base space. 

The main result of this note is the following formula for the total Chern class of the tangent bundle of a toric variety bundle. %TODO: add references to previous work on toric variety bundles
\begin{theorem*}[{\Cref{thm:chern_class_formula}}]
	Let \(\pi:E_\Sigma\to B\) be a toric variety bundle. Then, the total Chern class of the tangent bundle of \(E\) is given by the following formula.
	\[
		c(TE)=\pi^*c(TB)\cd \prod_{\rho\in \Sigma(1)}(1+[\wt{D_\rho}]).
	\]
	Here \([\wt{D_\rho}]\) is the cohomology class of the toric divisor associated to the ray \(\rho\) in the fan \(\Sigma\) corresponding to the toric variety bundle \(E_\Sigma\to X_\Xi\). Here the class \(\pi^*c(TB)\) denotes the pullback of the total Chern class of the tangent bundle of \(B\) along the map \(\pi\).
\end{theorem*}

The proof relies on the results of Khovanskii--Monin in the case of the toric base. By functoriality, the general case reduces to the case of universal bundles. As it turns out, universal toric variety bundles are fibered toric varieties. For fibered toric varieties, the above theorem is due to Khovanskii and Monin \cite[Proposition 3.7]{Khovanskii_Monin}. Consequently, the general case follows by a simple pullback argument.

Additionally, we show that this theorem generalizes to the case of unitary generalized quasitoric manifold bundles, which should be regarded as a topological version of toric variety bundles. In this case, the total Chern class of the tangent bundle of a unitary generalized quasitoric manifold bundle is given by a similar formula, where the classes of the characteristic submanifolds replace the toric divisor classes. 

We derive several interesting consequences of these results. The use of unitary generalized quasitoric manifold bundles allows us to give a new proof of Masuda's formula for the equivariant Chern classes of unitary generalized quasitoric manifolds \cite[Theorem 3.1]{Masuda}. In addition, we use our result to compute the total Chern class of a toroidal horospherical variety in terms of the total Chern class of the generalized flag variety and the total Chern class of the toric fibre. This calculation served as a starting point for the present work, as it provides a first step in a program for calculating invariants of subvarieties in spherical varieties, following the general philosophy of BKK theorems, see \cite{Kaveh_Khovanskii, Khovanskii_Monin_Limonchenko, Khovanskii_Monin} and references therein. We intend to explore this approach and its applications to calculating numeric invariants of spherical varieties following the general strategy of Danilov--Khovanskii \cite{Danilov_Khovanskii}.
\subsection*{Paper structure}
The paper is organized as follows. In \Cref{sec:toric_bundles}, we recall the definition of toric variety bundles and establish some basic properties. In \Cref{sec:chern_classes}, we prove the main result of this note. In \Cref{sec:quasitoric_bundles} we prove a version of the main result for unitary generalized quasitoric manifold bundles. Finally, in \Cref{sec:Concluding_Remarks}, we discuss applications and possible extensions of our results. In particular, we use our result to give an alternative proof of Masuda's formula for equivariant Chern classes of unitary generalized quasitoric manifolds (\Cref{subsec:equivariant_chern_classes}). In \Cref{subsec:toroidal_horospherical_varieties} we provide a calculation of Chern classes of toroidal horospherical varieties. Finally, in \Cref{subsec:euler_sequences}, we discuss a possible extension of our results to the case of singular toric variety bundles. To this end, we sketch an alternative proof of the main result using relative Euler sequences.
\subsection*{Acknowledgements}
We would like to thank Askold Khovanskii for posing the problem and providing helpful comments on the matter. We would also like to thank Leonid Monin for discussing the subject of this paper with us and reading its draft. We are also grateful to Taras Panov, for helpful comments. Finally, we are grateful to our partner, Alisa Chistopolskaya, for her constant support of our research.

\section{Toric Bundles}\label{sec:toric_bundles}
\subsection{Generalities on principal bundles}
\begin{definition}
	Let \(G\) be a topological group. A \textit{principal \(G\)-bundle} is a locally trivial bundle \(\pi: E \to B\) with a continuous right action of \(G\) on \(E\) such that the action is free and transitive on each fibre of \(\pi\).
\end{definition}
\begin{definition}
	Given a principal \(G\)-bundle \(\pi: E \to B\), and a left \(G\)-space \(X\), one can form a new bundle \(E \times_G X \to B\) with fibre \(X\) by taking the quotient of the product bundle \(E \times X\) by the diagonal action of \(G\).
\end{definition}
\begin{definition}
	Let \(G\) be a topological group. The \emph{classifying space} of \(G\), denoted by \(BG\), is defined as follows. \(BG\) is a topological space together with a principal \(G\)-bundle \(EG\to BG\) such that for any principal \(G\)-bundle \(E\) over any ``reasonable'' topological space \(B\) there exists a unique (up to homotopy) map \(f:B\to BG\) such that \(f^*EG\cong E\).
\end{definition}
The following example is well-known.
\begin{example}\label{ex:classifying_space_of_a_torus}
	Let \(G=(\CT)^n\). Then the classifying space \(BG\) is given by the infinite-dimensional projective space \((\CP^\infty)^n\). The universal bundle over \((\CP^\infty)^n\) is given by the infinite-dimensional vector bundle \((\C^\infty\setminus 0)^n\to (\CP^\infty)^n\). The fibres of this bundle are given by the punctured complex spaces \((\C^\infty\setminus 0)^n\).
\end{example}
\subsection{Toric variety bundles}
In this section, we define toric variety bundles and describe the ``universal toric variety bundle'' construction. For more details on toric variety bundles, we refer the reader to \cite[\S 2]{Khovanskii_Monin} and \cite[\S 3]{3_authors}.
\begin{construction}
	Let \(\Sigma\) be a fan in a lattice \(N\). Denote by \(X_\Sigma\) the corresponding toric variety, and denote by \(T_N\) the big algebraic torus acting on \(X_\Sigma\). Let \(\pi:E\to B\) be a principal \(T_N\)-bundle over a complex manifold \(B\). Then, the \emph{toric variety bundle} associated to \(E\) is the bundle \(E \times_{T_N} X_\Sigma  \to B\). We denote this bundle by \(\pi_\Sigma:E_\Sigma \to B\).
\end{construction}
\begin{proposition}\label{prop:toric_bundle_universal}
	Let \(N\) be a lattice and denote by \(T_N\) the associated algebraic torus \(N\otimes \CT\). The universal \(T_N\)-bundle is given by the following infinite-dimensional toric variety
	\[
		E_N=(\C^\infty\setminus 0)^n \to (\CP^\infty)^n=BT_N.
	\]	
	Consider a fan \(\Sigma\) in \(N\). Then the associated toric variety bundle \((E_N)_\Sigma\to BT_N\) is an infinite-dimensional fibered toric variety. 
\end{proposition}
\begin{proof}
	The assertion follows from \Cref{ex:classifying_space_of_a_torus} by observing that the map \((\C^\infty\setminus 0)^n \to (\CP^\infty)^n\) is algebraic and (by construction) toric. 
\end{proof}
\begin{proposition}
	Let \(E\to B\) be a principal \(T_N\)-bundle. Consider a fan \(\Sigma\) in \(N\). Then, the toric variety bundle \(E_\Sigma\to B\) is a pullback of the universal toric variety bundle \((E_N)_\Sigma\to BT_N\) along a map \(B\to BT_N\).
\end{proposition}
\begin{proof}
	This result is immediate from the similar result for \(T_N\)-bundles and the functoriality of the associated bundle construction. %TODO: fill in details
\end{proof}
\subsection{Fibered toric varieties}
Here, we briefly recall the concepts related to the special class of toric variety bundles known as fibered toric varieties. For more details, we refer the reader to \cite[\S 2]{Khovanskii_Monin}. We also refer the reader to \cite{CoxLittleSchenk} for a textbook introduction to toric varieties.
\begin{definition}[{\cite[Definition 2.1]{Khovanskii_Monin}}]
	Let \(\Sigma, \Xi\) be a pair of fans in lattices \(N, M\) respectively. We denote by \(\Phi:|\Xi|\to M\) a piecewise linear map, which is linear on each cone of \(\Xi\). A twisted fan \(\Xi\ltimes_\Phi \Sigma\) in the lattice \(M\oplus N\) is defined as follows.
	\[
		\Xi\ltimes_\Phi \Sigma=\{\wt{\sigma}+\tau \mid \sigma\in \Xi, \tau\in \Sigma\},\quad \wt{\sigma} \text{ is the graph of } \Phi\vert_\sigma.
	\]
\end{definition}
\begin{definition}
	Let \(E_\Sigma\to X_\Xi\) be a toric variety bundle over a toric variety \(X_\Xi\). We say that \(E_\Sigma\to X_\Xi\) is a \emph{fibered toric variety}.
\end{definition}
The following result of Khovanskii--Monin justifies this name.
\begin{theorem}[{\cite[Theorem 2.8]{Khovanskii_Monin}}]\label{thm:KM_twisted}
	Let \(E_\Sigma\to X_\Xi\) be a fibered toric variety. Then, the total space \(E_\Sigma\) is a toric variety defined by the twisted fan \(\Xi\ltimes_\Phi \Sigma\) for some piecewise linear map \(\Phi:|\Xi|\to M_\Sigma.\)
\end{theorem}
\begin{remark}
	The choice of the piecewise linear map \(\Phi\) is not unique. However, different choices of \(\Phi\) give rise to torically isomorphic toric varieties.
\end{remark}
\section{Chern Classes of Toric Variety Bundles}\label{sec:chern_classes}
\subsection{Vertical subbundles}
Below, we recall basic facts about vertical bundles for fibre bundles. The material is standard and can be found in \cite[\S 3]{Husemoller}.
\begin{construction}
	Let \(\pi: E\to B\) be a fibre bundle of complex manifolds. Then, the \emph{vertical tangent bundle} of \(E\) is the subbundle of the tangent bundle of \(E\) consisting of tangent vectors tangent to the fibres of \(E\). We denote this bundle by \(V(\pi)\). More precisely, the fibre of \(V(\pi)\) at a point \(e\in E\) is given by \(\ker d_e\pi\).
\end{construction}
\begin{remark}
	The fibre bundle condition on \(\pi\) ensures that \(V(\pi)\) is a subbundle, and not just a subsheaf, of the tangent bundle of \(E\).
\end{remark}
\begin{proposition}\label{prop:vertical_isomorphism_fiber_square}
	Consider a fibre square of the following form. Where \(B, E, B'\) are manifolds and \(\pi\) is a fibre bundle.
	% https://q.uiver.app/#q=WzAsNCxbMCwwLCJcXHRhdV4qRT1FJyJdLFsyLDAsIkUiXSxbMCwyLCJCJyJdLFsyLDIsIkIiXSxbMCwxLCJcXHd0e1xcdGF1fSJdLFswLDIsIlxcdGF1XipcXHBpIiwyXSxbMiwzLCJcXHRhdSIsMl0sWzEsMywiXFxwaSJdLFswLDMsIiIsMSx7InN0eWxlIjp7Im5hbWUiOiJjb3JuZXIifX1dXQ==
\[\begin{tikzcd}
	{\tau^*E=E'} && E \\
	\\
	{B'} && B
	\arrow["{\wt{\tau}}", from=1-1, to=1-3]
	\arrow["{\tau^*\pi}"', from=1-1, to=3-1]
	\arrow["\lrcorner"{anchor=center, pos=0.125}, draw=none, from=1-1, to=3-3]
	\arrow["\pi", from=1-3, to=3-3]
	\arrow["\tau"', from=3-1, to=3-3]
\end{tikzcd}\]
	Then, we have the following natural isomorphism of vertical subbundles.
	\[
		\tau^*V(\pi)\cong V(\tau^*\pi).
	\]
\end{proposition}
\begin{proof}
	The result immediately follows from the fact that the vertical subbundle is defined as the subbundle of the tangent bundle spanned by the vectors tangent to the fibres of the fibre bundle, and the pullback preserves fibres.
\end{proof}
The following result is immediate and can be taken as a definition of the vertical bundle.
\begin{proposition}\label{prop:vertical_sequence}
	Given a fibre bundle \(\pi:E\to B\) of complex manifolds, there is the following exact sequence of vector bundles on \(E\).
	\begin{equation}\label{eq:vertical_sequence}
		0\to V(\pi)\to TE\to \pi^*TB\to 0.
	\end{equation}
\end{proposition}
The following corollary is an immediate consequence of the above result and the Whitney formula for the total Chern class.
\begin{corollary}
	There is the following identity of total Chern classes induced by the sequence \eqref{eq:vertical_sequence}.
	\[
		c(TE)=c(V(\pi))c(\pi^*TB).
	\]
\end{corollary}
\subsection{Chern classes of complete toric varieties}
\begin{construction}
	Let \(\Sigma\) be a complete polyhedral fan in a lattice \(N\). Denote by \(X_\Sigma\) the corresponding toric variety. Denote by \(\Sigma(1)\) the set of rays of \(\Sigma\). The \emph{toric divisor} associated to a ray \(\rho\in \Sigma(1)\) is defined as the closure of the corresponding torus orbit in \(X_\Sigma\). We denote this divisor by \(D_\rho\). Accordingly, we denote by \([D_\rho]\) the class of the divisor \(D_\rho\) in the second cohomology group of \(X_\Sigma\) with integer coefficients.
\end{construction}
The following result is standard. A textbook reference is \cite[Proposition 13.1.2]{CoxLittleSchenk}.
\begin{theorem}
	Let \(X_\Sigma\) be a complete smooth toric variety. Then, the total Chern class of the tangent bundle of \(X_\Sigma\) is given by the following formula.
	\[
		c(TX_\Sigma)=\prod_{\rho\in \Sigma(1)}(1+[D_{\rho}]).
	\]
\end{theorem}
\Cref{thm:KM_twisted} implies the following description of the total Chern class of a fibered toric variety.
\begin{proposition}[{\cite[Proposition 3.7]{Khovanskii_Monin}}]
	Let \(E_\Sigma\to X_\Xi\) be a fibered toric variety. Then, the total Chern class of the tangent bundle of \(E_\Sigma\) is given by the following formula.
	\[
		c(TE_\Sigma)=\prod_{\rho\in \Xi(1)}(1+[\wt{D_\rho}])\cd \prod_{\rho\in \Sigma(1)}(1+[D_{\rho}]).
	\]
\end{proposition}
An immediate corollary of this result is the following formula for the total Chern class of the vertical subbundle of a fibered toric variety.
\begin{corollary}\label{cor:vertical_chern_class_fibered_case}
	Let \(\pi_\Xi:E_\Xi\to X_\Sigma\) be a fibered toric variety. Then, the total Chern class of the vertical subbundle of the tangent bundle of \(E\) is given by the following formula.
	\[
		c(V(\pi_\Xi))=\prod_{\rho\in \Sigma(1)}(1+[\wt{D_\rho}]).
	\]
	Here \([\wt{D_\rho}]\) is the cohomology class of the toric divisor associated to the ray \(\rho\) in the fan \(\Sigma\) viewed as a cohomology class in the second cohomology group of \(E_\Xi\) with integer coefficients.
\end{corollary}
\subsection{Khovanskii--Monin conjecture}
In \cite[Conjecture 3.6]{Khovanskii_Monin}, Khovanskii and Monin conjectured the following formula for the total Chern class of the tangent bundle of a toric variety bundle.
\begin{theorem}\label{thm:chern_class_formula}
	Let \(\pi:E_\Sigma\to B\) be a toric variety bundle. Then, the total Chern class of the tangent bundle of \(E\) is given by the following formula.
	\[
		c(TE)=\pi^*c(TB)\cd \prod_{\rho\in \Sigma(1)}(1+[\wt{D_{\rho}}]).
	\]
\end{theorem}
\begin{proof}
	This result is a consequence of \Cref{prop:toric_bundle_universal,prop:vertical_isomorphism_fiber_square,prop:vertical_sequence}. Indeed, by \Cref{prop:toric_bundle_universal} the universal toric variety bundle of type \((N,\Sigma)\) is a fibered toric variety, hence we have the following formula for the total Chern class of the vertical subbundle by \Cref{cor:vertical_chern_class_fibered_case}.
	\[
		c(V((\pi_N)_\Sigma))=\prod_{\rho\in \Sigma(1)}(1+[\wt{D_{\rho}}]).
	\]
	By \Cref{prop:vertical_isomorphism_fiber_square}, the vertical subbundle of the toric variety bundle \(E_\Sigma\) is isomorphic to the pullback of the vertical subbundle of the universal toric variety bundle along the map \(B\to BT_N\). By \Cref{prop:vertical_sequence}, the total Chern class of the tangent bundle of \(E_\Sigma\) is given by the product of the total Chern class of the vertical subbundle and the total Chern class of the pullback. The result follows.
\end{proof}
\section{Chern Classes of Unitary Generalized Quasitoric Manifold Bundles}\label{sec:quasitoric_bundles}
\subsection{Generalized quasitoric manifolds}
Below, we provide a brief summary of the definition and construction of generalized quasitoric manifolds. For more details, see \cite[\S 2]{Khovanskii_Monin_Limonchenko} and \cite[\S 7.3]{BuPa}.

Let \(K\) be an abstract simplicial complex on a set \(S\). Let \(\DD\) and \(\SS\) be respectively the unit disk and the unit circle in \(\C\). Then we can define the moment-angle complex associated to \(K\) as follows.
\[
	Z(K)=\bigcup_{\sigma\in K} \DD^\sigma \times \SS^{S\setminus \sigma}\sse \C^S.
\] 

We assume now that \(K\) arises as a cone complex of a complete simplicial fan \(\Sigma\) in \(\R^n\). Denote by \(N\) the cocharacter lattice of the compact real torus \(T^n.\)
Assume that there is a \emph{characteristic map} \(S=\Sigma(1)\to N\) such that whenever \(\sigma\in K\) is a face the collection \(\{\Lambda(s)\mid s\in \sigma\}\) can be extended to a basis of the lattice \(N\). By abuse of notation we denote by \(\Lambda\) also the map it induces from \(T^S\to T^n\) by sending \(e_s\to \Lambda(s).\) In this case, we can define the \emph{generalized quasitoric manifold bundle} associated to \(K\) as follows.
\[
	X_{\Sigma,\Lambda}=Z(K)/\exp(\ker \Lambda).
\]

Then, under the additional assumption that the fan \(\Sigma\) is star-shaped, the manifold \(X_{\Sigma,\Lambda}\) is a smooth manifold. Clearly, a generalized quasitoric manifold admits a residual action of the compact torus \(T^n\) on it.
\begin{construction}
	Given a principal \(T^n\)-bundle \(\pi:E\to B\) and a generalized quasitoric manifold \(X_{\Sigma,\Lambda}\), we can form the associated bundle \(E\times_{T^n} X_{\Sigma,\Lambda}\to B\). We denote this bundle by \(\pi_\Sigma:E_\Sigma\to B\).
\end{construction}

\begin{definition}
	When the generalized quasitoric manifold \(X_{\Sigma,\Lambda}\) has an equivariant stably complex structure, we say, following \cite{Masuda}, that \(X_{\Sigma,\Lambda}\) is a \emph{unitary generalized quasitoric manifold}. In this case, we say that the associated bundle \(\pi_\Sigma:E_\Sigma\to B\) is a \emph{unitary generalized quasitoric manifold bundle}.
\end{definition}
\begin{notation}
	 To somewhat simplify the notation, we will abbreviate the term \emph{unitary generalized quasitoric} manifold (respectively bundle) to \UGQT-manifold (respectively bundle).
\end{notation}
\begin{lemma}\label{prop:stable_complexity_of_quasitoric_manifold_bundles}
	Let \(\pi:E_{\Sigma,\Lambda}\to B\) be a \UGQT-manifold bundle over a base space \(B\). Then, it admits a fibrewise stably complex structure equivariant for the torus action on the fibres. Moreover, if \(B\) is a stably-complex manifold, then \(E_\Sigma\) is a stably-complex manifold.
\end{lemma}
\begin{proof}
	The result is immediate, since the vertical subbundle of a \UGQT-manifold bundle is a stably-complex vector bundle by assumption. The second assertion follows from the exact sequence of vector bundles associated to the fibre bundle \(\pi:E_{\Sigma,\Lambda}\to B\) given by \Cref{prop:vertical_sequence}. The vertical subbundle is stably complex, and the pullback of the tangent bundle of \(B\) is stably complex by assumption. Thus, the total space \(E_{\Sigma,\Lambda}\) is stably complex.
\end{proof}
\begin{remark}
	By \cite[Corollary 7.3.16]{BuPa} in the case when the fan \(\Sigma\) is the normal fan of a simple convex polytope \(P\) in \(\R^n\), the generalized quasitoric manifold \(X_{\Sigma,\Lambda}\) is a \UGQT-manifold. In this case, the associated bundle \(\pi_\Sigma:E_\Sigma\to B\) is a \UGQT-bundle. However, it is not immediately clear whether any generalized quasitoric manifold is a \UGQT-manifold.
\end{remark}
\subsection{Fibered generalized quasitoric manifolds}
\begin{definition}
	A fibered generalized quasitoric manifold is a generalized quasitoric manifold bundle \(\pi:E_{\Sigma,\Lambda}\to X_{\Xi,\Omega}\) such that the base space \(X_{\Xi,\Omega}\) is a generalized quasitoric manifold and the total space \(E_{\Sigma,\Lambda}\) is a generalized quasitoric manifold. Additionally, the map \(\pi\) is equivariant for the torus action on the base and the total space.
\end{definition}
\begin{definition}
	Let \((\Sigma,\Lambda)\) be a pair of a fan \(\Sigma\) in a lattice \(N\) and a characteristic map \(\Lambda:\Sigma(1)\to N\). Let \((\Xi,\Omega)\) be a pair of a fan \(\Xi\) in a lattice \(M\) and a characteristic map \(\Omega:\Xi(1)\to M\). For \(\Phi:|\Xi|\to M_\Sigma\) a piecewise linear map, linear on each cone of \(\Xi\). We define the \emph{twisted pair} \((\Xi\ltimes_\Phi \Sigma,\Omega\times_\Phi \Xi)\) as follows. By definition, the twisted pair is given by the following data. 
\begin{itemize}
	\item The fan \(\Xi\ltimes_\Phi \Sigma\) in the lattice \(M\oplus N\) is defined as follows.
	\[
		\Xi\ltimes_\Phi \Sigma=\{\wt{\sigma}+\tau \mid \sigma\in \Xi, \tau\in \Sigma\},\quad \wt{\sigma} \text{ is the graph of } \Phi\vert_\sigma.
	\]
	\item The characteristic map \(\Omega\times_\Phi \Lambda:\Xi(1)\times \Sigma(1)\to M_\Sigma\) is defined as follows.
	\[		
		(\Omega\times_\Phi \Lambda)\vert_{\Sigma(1)}=\Lambda,\quad (\Omega\times_\Phi \Lambda)\vert_{\Xi(1)}=\Phi\circ \Omega.
	\]
\end{itemize}
\end{definition}
\begin{theorem}\label{thm:KM_twisted_for_generalized_quasitoric_manifolds}
	Let \(E_{\Sigma,\Lambda} \to X_{\Xi,\Omega}\) be a generalized quasitoric manifold bundle over generalized quasitoric manifold \(X_{\Xi,\Omega}\). Then, it is equivariantly diffeomorphic to a fibered quasitoric manifold determined by the twisted fan \(\Xi\ltimes_\Phi \Sigma\) and the characteristic map \(\Lambda\times \Omega\). Here \(\Phi:|\Xi|\to M_\Sigma\) is a piecewise linear map, linear on each cone of \(\Xi\).
\end{theorem}
The above result is proved similarly to \cite[Theorem 2.8]{Khovanskii_Monin} and so we only sketch the proof. The key idea is to prove equivariance of the principal bundle \(E\to X_{\Xi,\Omega}\) by averaging over the torus acting on the base and then use the combinatorial description of equivariant morphisms between generalized quasitoric manifolds.
\begin{proof}[Proof sketch]
	Given a principal \(T^n\)-bundle \(\pi:E\to X_{\Xi,\Omega}\), we can associate to it a split complex \(n\)-dimensional vector bundle \(L_1\oplus\ldots \oplus L_n\) over \(X_{\Xi,\Omega}\) with a fiberwise action of the torus~\(T^n\). Denote by \(T^k\) the half-dimensional torus acting on \(X_{\Xi,\Omega}.\) Each \(L_i\) is isomorphic to a \(T^k\)-equivariant line bundle because the morphism from the degree two \(T^k\)-equivariant cohomology to the degree two ordinary cohomology is surjective.
	
	By a standard averaging argument, there exists a \(T^k\)-invariant metric on \(L_1\oplus\ldots \oplus L_n\) such that the action of \(T^k\) on the fibers of \(L_1\oplus\ldots \oplus L_n\) is unitary. Consequently, the principal \(T^n\)-bundle \(\pi:E\to X_{\Xi,\Omega}\) can be viewed as the unitary principal bundle associated to the equivariant vector bundle \(L_1\oplus\ldots \oplus L_n\). Consequently, there is a \(T^k\times T^n\)-action on the total space \(E_{\Sigma,\Lambda}\) making the bundle \(\pi:E_{\Sigma,\Lambda}\to X_{\Xi,\Omega}\) a fibered generalized quasitoric manifold. It remains to observe that the pair associated to the manifold \(E_{\Sigma,\Lambda}\) is precisely the twisted pair \((\Xi\ltimes_\Phi \Sigma,\Omega\times_\Phi \Lambda)\) for some piecewise linear map \(\Phi:|\Xi|\to M_\Sigma\) linear on each cone of \(\Xi\) by a direct calculation.
\end{proof}
A crucial observation is that when we apply the associated bundle construction to the universal bundle
\[
	E\xrightarrow{T^n} (\CP^\infty)^n
\]
we get a fibered generalized quasitoric manifold. More precisely, we have a colimit of fibered generalized quasitoric manifolds. We record this observation in the following corollary.
\begin{corollary}\label{cor:universal_generalized_quasitoric_manifold_bundle}
	The universal generalized quasitoric manifold bundle \({\pi:E_{\Sigma,\Lambda}\to (\CP^\infty)^n}\) is a fibered generalized quasitoric manifold. 
\end{corollary}
\subsection{Formula for the total Chern class}
\begin{proposition}
	The total Chern class of a fibered \UGQT-manifold \({\pi:E_{\Sigma,\Lambda}\to X_{\Xi,\Omega}}\) is given by the following formula.
	\[	c(TE_{\Sigma,\Lambda})=\prod_{\rho\in \Sigma(1)\cup \Xi(1)}(1+[\wt{D_{\rho}}]).\]
	Here \([\wt{D_{\rho}}]\) are the cohomology classes induced from the fibre and the base quasitoric manifolds. 
\end{proposition}
\begin{proof}
	The result follows from the description of the total Chern class of a quasitoric manifold, see \cite[Theorem 7.3.28]{BuPa} and the fact that \(E_{\Sigma,\Lambda}\) is a quasitoric manifold by \Cref{thm:KM_twisted_for_generalized_quasitoric_manifolds}.
\end{proof}
\begin{corollary}
	The total Chern class of the vertical subbundle of a \UGQT-manifold bundle \(\pi:E_\Sigma\to B\) is given by the following formula.
	\[	c(V(\pi))=\prod_{\rho\in \Sigma(1)}(1+[\wt{D_{\rho}}]).\]
	Here \([\wt{D_{\rho}}]\) is the cohomology class induced in \(E_{\Sigma,\Lambda}\) by the characteristic manifold associated to the ray \(\rho.\)
\end{corollary}
\begin{proof}
	The proof is identical to the proof given in \Cref{cor:vertical_chern_class_fibered_case}.
\end{proof}
\begin{theorem}\label{thm:chern_class_quasitoric}
	The total Chern class of the tangent bundle of a \UGQT-manifold bundle \(\pi:E_{\Sigma,\Lambda}\to B\) over a stably complex manifold is given by the following formula.
	\[
		c(TE_{\Sigma,\Lambda})=\pi^*c(TB)\cd \prod_{\rho\in \Sigma(1)}(1+[\wt{D_{\rho}}]).
	\]
\end{theorem}
\begin{proof}
	The proof is identical to the proof of \Cref{thm:chern_class_formula}. The only difference is we view \((\CP^\infty)^n\) as the classifying space for \(T^n\) instead of \((\CT)^n\). Thus by applying the same argument together with the existence of the stably complex structure on \(E_{\Sigma,\Lambda}\) afforded by \Cref{prop:stable_complexity_of_quasitoric_manifold_bundles} we obtain the result.
\end{proof}
\begin{remark}
	Note that \Cref{thm:chern_class_formula} is not a special case of the above theorem. The reason is that here we assume \(\Sigma\) arises from a polytope, whereas in \Cref{thm:chern_class_formula} we only assume \(\Sigma\) is a fan that determines a smooth toric variety. 
\end{remark}
\section{Concluding Remarks}\label{sec:Concluding_Remarks}
\subsection{Equivariant Chern classes of toric varieties and quasitoric manifolds}\label{subsec:equivariant_chern_classes}
In this section, we provide an alternative proof of Masuda's description of equivariant Chern classes for unitary generalized quasitoric manifolds. The calculation is based on the same idea as the calculation of the cohomology of toric varieties given in \cite[\S 4]{Khovanskii_Monin}. The key idea is to view the Borel construction of a generalized quasitoric manifold as a fibered generalized quasitoric manifold. The main advantage of this approach is that it reduces the equivariant problem to the non-equivariant one and provides a geometric interpretation of the equivariant result.

We have the following general result for the total equivariant Chern class of a generalized quasitoric manifold with an equivariant stably complex structure. 
\begin{theorem}[{\cite[Theorem 3.1]{Masuda}}]
	Let \(X\) be a manifold with a half-dimensional compact torus \(T\) action with isolated fixed point set and equivariant stably complex structure. Denote by \(\Sigma(1)\) the set of codimension \(2\) \(T\)-invariant submanifolds. Then, the total equivariant Chern class of the tangent bundle of \(X\) is given (modulo \(H_T^*(BT)\)-torsion) by the following formula.
	\[
		c^T(TX_{\Sigma,\Lambda})=\prod_{\rho\in \Sigma(1)}(1+[D_{\rho}]_T)
	\]
	Here \([D_{\rho}]_T\) is the equivariant cohomology class associated to the invariant submanifold \(\rho\in \Sigma(1).\)
\end{theorem}
Now consider a generalized quasitoric manifold \(X_{\Sigma,\Lambda}\) with an equivariant stably complex structure. Denote the half-dimensional compact torus acting on \(X_{\Sigma,\Lambda}\). Then we can form the Borel construction \((X_{\Sigma,\Lambda})_T=X_{\Sigma,\Lambda}\times_T ET\). The total space of this construction is a fibered generalized quasitoric manifold over the classifying space \(BT\) with fibre \(X_{\Sigma,\Lambda}\). Thus, we can apply \Cref{thm:chern_class_quasitoric} to obtain a formula for the total equivariant Chern class of the Borel construction. 

To make the argument slightly more psychologically comfortable, one may choose to work with a sequence of finite-dimensional approximations of the Borel construction, given by, for instance, \[\left(S^{2N+1}\right)^n\to \left(\CP^N\right)^n.\]However, we will not do this here for the sake of brevity. We have the following result, which immediately follows from \Cref{thm:chern_class_quasitoric}.
\begin{proposition}\label{thm:chern_equiv_basic}
	Let \(X_{\Sigma,\Lambda}\) be a \UGQT-manifold. Then, the total Chern class of the vertical bundle of the projection \((X_{\Sigma,\Lambda})_T\to BT\) is given by the following formula.
	\[
		c(V(\pi))=\prod_{\rho\in \Sigma(1)}(1+[D_{\rho}]_T).
	\]
	Here \([D_{\rho}]_T\) is the equivariant cohomology class of the characteristic submanifold associated to the ray \(\rho\) in the fan \(\Sigma\).
\end{proposition}
\begin{lemma}\label{lem:Borel_of_Tangent}
	Given a \(G\)-manifold \(M\) for some topological group, the Borel construction \(\left(TM\right)_G\) is isomorphic to the vertical subbundle of the projection \((M)_G\to BG\).
\end{lemma}
\begin{proof}
	Consider the following diagram of Cartesian squares.
	% https://q.uiver.app/#q=WzAsNixbMCwwLCJUTT1WKFxccGkpIl0sWzIsMCwiVihcXHBpX0cpIl0sWzIsMiwiTV9HIl0sWzIsNCwiQkciXSxbMCw0LCIqIl0sWzAsMiwiTSJdLFswLDFdLFsxLDJdLFsyLDMsIlxccGlfRyJdLFs0LDNdLFs1LDQsIlxccGkiLDJdLFs1LDJdLFswLDVdLFswLDIsIiIsMCx7InN0eWxlIjp7Im5hbWUiOiJjb3JuZXIifX1dLFs1LDMsIiIsMCx7InN0eWxlIjp7Im5hbWUiOiJjb3JuZXIifX1dXQ==
	\[\begin{tikzcd}[ampersand replacement=\&]
		{TM=V(\pi)} \&\& {V(\pi_G)} \\
		\\
		M \&\& {M_G} \\
		\\
		{*} \&\& BG
		\arrow[from=1-1, to=1-3]
		\arrow[from=1-1, to=3-1]
		\arrow["\lrcorner"{anchor=center, pos=0.125}, draw=none, from=1-1, to=3-3]
		\arrow[from=1-3, to=3-3]
		\arrow[from=3-1, to=3-3]
		\arrow["\pi"', from=3-1, to=5-1]
		\arrow["\lrcorner"{anchor=center, pos=0.125}, draw=none, from=3-1, to=5-3]
		\arrow["{\pi_G}", from=3-3, to=5-3]
		\arrow[from=5-1, to=5-3]
	\end{tikzcd}\]
	By the universal property, it follows that the vertical subbundle of the projection \((M)_G\to BG\) is isomorphic to the Borel construction of the tangent bundle of \(M\). 
\end{proof}
Combining this result with the description of equivariant cohomology of quasitoric manifolds, we get the following result, which is a version of Masuda's result \cite[Theorem 3.1]{Masuda} for unitary generalized quasitoric manifolds.
\begin{theorem}\label{thm:chern_equiv_quasitoric}
	The total equivariant Chern class of \UGQT-manifold \(X_{\Sigma,\Lambda}\) is given by the following formula.
	\[	c^T(TX_{\Sigma,\Lambda})=\prod_{\rho\in \Sigma(1)}(1+[D_{\rho}]_T).\]
	Here \([D_{\rho}]_T\) is the equivariant cohomology class of the characteristic submanifold associated to the ray \(\rho\) in the fan \(\Sigma\).
\end{theorem}
\begin{proof}
	By \Cref{lem:Borel_of_Tangent} the Borel construction of the tangent bundle of \(X_{\Sigma,\Lambda}\) is isomorphic to the vertical subbundle of the projection \((X_{\Sigma,\Lambda})_T\to BT\). Thus, we can apply \Cref{thm:chern_equiv_basic} to obtain the result.
\end{proof}
Note that the above result holds on the nose without the need to mod out the \(H_T^*(BT)\)-torsion. This still, however, follows from Masuda's result since the equivariant cohomology of a \UGQT-manifold is \(H^*(BT)\)-torsion-free.
\subsection{Total Chern class of a toroidal horospherical variety}\label{subsec:toroidal_horospherical_varieties}
In a broader context, the results of this paper provide a way to calculate the Chern classes of \emph{toroidal horospherical varieties} in terms of the Chern classes of the base space and the Chern classes of the toric fibre. We describe this briefly below. For a general introduction to horospherical varieties, we refer the reader to the book \cite{Timashev} for a comprehensive introduction to the subject. We also refer the reader to \cite{Terpereau} for a shorter introduction.
\begin{definition}[{\cite[Definition 7.1]{Timashev}}]
Let $G$ be a complex Lie group. A closed subgroup $H \subset G$ is called a \emph{horospherical subgroup} if $H$ contains a maximal unipotent subgroup \(U\) of $G$. A $ G$-variety $X$ is called a \emph{horospherical variety} if the stabilizer of any point in $X$ is a horospherical subgroup. Equivalently, $X$ can be written as $X = GX^U$.
\end{definition}
Denote by \(P=N_G(H)\) the normalizer of a horospherical subgroup \(H\) in a complex Lie group \(G\). 
\begin{lemma}[{\cite[Satz 2.1]{Knop}, \cite[Lemma 7.4]{Timashev}}]
	Let \(G\) be a complex algebraic group and \(H\) a horospherical subgroup of \(G\). Then, the normalizer \(P\) of \(H\) in \(G\) is a parabolic subgroup of \(G\). 
\end{lemma}
The following result is an immediate consequence of the above lemma.
\begin{corollary}
	The quotient \(G/N_G(P)\) is a projective variety. 
\end{corollary}
\begin{construction}
	Given a horospherical variety \(X\), there is a \(G\)-equivariant canonical map \(\vp:X\to G/P\). Denote by \(F\) the fibre of this map over \([e]\in G/P\).
\end{construction}
\begin{definition}[{\cite[\S 1.2]{Terpereau}}]
	A variety \(X\) is called a \emph{toroidal horospherical variety} if it verifies the following condition.
	\[
		X\cong G\times_{G/P} F.
	\]
	In this case, the fibre \(F\) has a natural structure of a toric variety. As a result, the variety \(X\) has a natural structure of a toric variety bundle over \(G/P\).
\end{definition}
There is also a characterization of toroidal horospherical varieties as varieties having no colours in the sense of Luna--Vust \cite[\S 29]{Timashev}.
We have the following version of \Cref{thm:chern_class_formula} for toroidal horospherical varieties.
\begin{theorem}
	Let \(X\) be a toroidal horospherical variety. Then, the total Chern class of the tangent bundle of \(X\) is given by the following formula.
	\[
		c(TX)=\vp^*c(T(G/P))\cd \prod_{\rho\in \Sigma_F(1)}(1+[\wt{D_\rho}]).
	\]
	Here, \(P\) is the normalizer of the horospherical subgroup \(H\) in \(G\), \(\Sigma_F\) is the fan associated to the toric fibre \(F\) of the map \(\vp:X\to G/P\).
\end{theorem}
\subsection{Euler sequences for toric variety bundles}\label{subsec:euler_sequences}
In this section, we discuss an alternative approach to computing the Chern classes of toric variety bundles. The main idea is to use the Euler sequence for toric varieties. The Euler sequence is a short exact sequence of vector bundles on a toric variety \(X_\Sigma\) generalizing the classical Euler sequence for the projective space. Classically, it is used to compute the Chern classes of the tangent bundle of a toric variety, see \cite[\S 13.1]{CoxLittleSchenk}. The main advantage of this approach is that it likely admits a generalization to the case of singular toric varieties. The main disadvantage is that it is more technical than the approach used in this paper. 
\begin{proposition}[{\cite[Example 13.1.1]{CoxLittleSchenk}}]
    Let $X = X_\Sigma$ be a smooth complete toric variety. Its cotangent
    bundle $\Omega_X^1$ fits into the generalized Euler sequence
    \begin{equation*}
        0 \longrightarrow \Omega_X^1 \longrightarrow \bigoplus_{\rho\Sigma(1)} \mathcal{O}_X(-D_\rho) \longrightarrow \mathrm{Pic}(X) \otimes_\mathbb{Z} \mathcal{O}_X \longrightarrow 0
    \end{equation*}
\end{proposition}
\begin{construction}\label{con:pull-push}
	By definition, a toric variety bundle is constructed as the following quotient of the product of a principal \(T_N\)-bundle and a toric variety \(X_\Sigma\) for the torus \(T_N\).
	\[
		E_\Sigma=E\times_{T_N} X_\Sigma.
	\]
	Thus given a line bundle \(\Lc\to X_\Sigma\) on \(X_\Sigma\), we can pull it back to the product \(E\times X_\Sigma\). This gives rise to an \emph{equivariant line bundle} over \(E\times X_\Sigma\). Consequently, it pushes forward to a line bundle on the quotient \(E_\Sigma\). We denote this line bundle by~\(\Lc_{E_\Sigma}\).
\end{construction}
Intuitively, the line bundle \(\Lc_{E_\Sigma}\) is obtained by considering the line bundle \(\Lc\) as a line bundle on \(E_\Sigma\) by considering a fibre isomorphic to \(X_\Sigma\) through every point of \(E_\Sigma\). 
\begin{proposition}
	Take a fibered toric variety \(E_\Sigma\to B\) with the toric fibre \(X_\Sigma\) corresponding to a fan \(\Sigma\), then we have the following exact sequence of coherent sheaves on \(X\). 
	\[
		0\to \Omega_{E/B}^1\to \bigoplus_{\rho\in \Sigma(1)} \mathcal{O}_E(-D_\rho)\to \mathrm{Pic}(X_\Sigma) \otimes_\mathbb{Z} \mathcal{O}_E\to 0. 
	\]
	Here \(\O_{E_\Sigma}(-D_\rho)\) is the line bundle on \(E_\Sigma\) associated with the line bundle \(\O_{X_\Sigma}(-D_\rho)\) on \(X_\Sigma\) by means of \Cref{con:pull-push}.
\end{proposition}
\begin{proof}[Sketch of the proof]
	The result is identical to the proof of \cite[Theorem 8.6]{CoxLittleSchenk} if one passes to sheaves of relative forms. 
\end{proof}
\begin{corollary}\label{cor:vertical_chern_class_euler}
	Let \(E_\Sigma\to B\) be a toric variety bundle. Then, the total Chern class of the vertical subbundle of the tangent bundle of \(E_\Sigma\) is given by the following formula.
	\[
		c(V(\pi))=\prod_{\rho\in \Sigma(1)}(1+[\wt{D_\rho}]).
	\]
\end{corollary}
\begin{proof}
	The result is immediate by functoriality of cohomology and the fact that \(V(\pi)\) is the dual bundle of the cotangent vertical bundle \(\Omega_{E_\Sigma/B}^1\).
\end{proof}
Using the above result, we can compute the Chern classes of the toric variety bundle \(E_\Sigma\to B\) as follows.
\begin{theorem}
	Let \(E_\Sigma\to B\) be a toric variety bundle. Then, the total Chern class of the tangent bundle of \(E_\Sigma\) is given by the following formula.
	\[
		c(TE)=\pi^*c(TB)\cd \prod_{\rho\in \Sigma(1)}(1+[\wt{D_\rho}]).
	\]
\end{theorem}
\begin{proof}
	By \Cref{prop:vertical_sequence} we have the following exact sequence of vector bundles on the manifold \(E_\Sigma\).
	\[
		0\to V(\pi)\to TE\to \pi^*TB\to 0.
	\]
	As a result, the theorem follows from the Whitney formula and \Cref{cor:vertical_chern_class_euler}.
\end{proof}
It should be possible to extend the above result to the case of singular toric varieties. The main idea is to use the theory developed by MacPherson for the Chern classes of singular varieties (see \cite{MacPherson}) and the Euler sequence as above to compute the Chern classes of the vertical subbundle. We intend to explore this direction in future work.


\begin{thebibliography}{10}
\bibitem[BP2015]{BuPa} Buchstaber, Victor M., and Taras E. Panov. Toric Topology. Mathematical Surveys and Monographs, volume 204. Providence: American mathematical society, 2015. \doi{10.1090/surv/204}.
\bibitem[CLS2011]{CoxLittleSchenk} 
    Cox, David, John Little, and Henry Schenck. Toric Varieties. Graduate Studies in Mathematics, July 7, 2011. \doi{10.1090/gsm/124}.
\bibitem[DK1987]{Danilov_Khovanskii} Danilov, V.~I.; Khovanskii, A.~G. Newton polyhedra and an algorithm for computing Hodge--Deligne numbers. Mathematics of the USSR-Izvestiya 29 (1987), no. 2, 279--298. \doi{10.1070/IM1987v029n02ABEH000970}.
\bibitem[HKM2021]{3_authors} Hofscheier, Johannes, Askold Khovanskii, and Leonid Monin. “Cohomology Rings of Toric Bundles and the Ring of Conditions.” arXiv, April 20, 2021. \arXiv{2006.12043}.
\bibitem[Hu1994]{Husemoller} Husemoller, Dale. Fibre Bundles. Graduate Texts in Mathematics. New York, NY: Springer New York, 1994. \doi{10.1007/978-1-4757-2261-1}.
\bibitem[KK2012]{Kaveh_Khovanskii} K. Kaveh and A.~G. Khovanskii, Algebraic equations and convex bodies, in {\it Perspectives in analysis, geometry, and topology}, 263--282, Progr. Math., 296, Birkh\"auser/Springer, New York, 2012; \doi{10.1007/978-0-8176-8277-4\_12}.
\bibitem[KM2023]{Khovanskii_Monin} Khovanskii, Askold; Monin, Leonid. Fibered toric varieties. Mosc. Math. J. 23 (2023), no. 4, 545--558. \arXiv{2311.01754}.
\bibitem[KML2021]{Khovanskii_Monin_Limonchenko} Khovanskii, Askold, Ivan Limonchenko, and Leonid Monin. “Cohomology Rings of Quasitoric Bundles.” arXiv, December 30, 2021. \arXiv{2112.14970}.

\bibitem[Kn1990]{Knop} Knop, Friedrich. “Weylgruppe und Momentabbildung.” Inventiones mathematicae 99, no. 1 (December 1, 1990): 1–23. \doi{10.1007/BF01234409}.
\bibitem[McP1974]{MacPherson} MacPherson, R. D. “Chern Classes for Singular Algebraic Varieties.” Annals of Mathematics 100, no. 2 (1974): 423–32. \doi{10.2307/1971080}.
\bibitem[Ma1999]{Masuda} Masuda, Mikiya. “Unitary Toric Manifolds, Multi-Fans and Equivariant Index.” Tohoku Mathematical Journal 51, no. 2 (January 1, 1999). \doi{10.2748/tmj/1178224815}.
\bibitem[Te2019]{Terpereau} Terpereau, Ronan, ``An Overview of the Classification of Spherical and Complexity-One Varieties'', Lecture Notes, \web{http://ronan.terpereau.perso.math.cnrs.fr/enseignement/overview_classification_spherical_and_compelxity-one_varieties.pdf}.
\bibitem[Ti2011]{Timashev} Timashev, D.A. Homogeneous Spaces and Equivariant Embeddings. Encyclopaedia of Mathematical Sciences. Berlin, Heidelberg: Springer Berlin Heidelberg, 2011. \doi{10.1007/978-3-642-18399-7}.


\end{thebibliography}
\end{document}